\documentclass[11pt,reqno]{amsart}

\usepackage{graphicx}
\usepackage{mathrsfs}
\usepackage{color}
\usepackage{comment}
\usepackage{amssymb}
\usepackage{enumitem}

\usepackage[margin = 1.4in] {geometry}

\usepackage[hyperindex,breaklinks]{hyperref}

\allowdisplaybreaks

\newtheorem{prop}{Proposition}%[section]
\newtheorem{theo}[prop]{Theorem}
\newtheorem*{theo*}{Theorem}
\newtheorem{lemm}[prop]{Lemma}
\newtheorem{coro}[prop]{Corollary}

\theoremstyle{definition}
\newtheorem{rema}[prop]{Remark}

\newtheorem{quest}[prop]{Question}

\newcommand{\RR}{\mathbf{R}}

\newcommand{\bphi}{\mathbf{\Phi}}

\newcommand{\cA}{\mathcal A}

\newcommand{\cH}{\mathcal H}

\newcommand{\cQ}{\mathcal Q}

\DeclareMathOperator{\tr}{tr}

\DeclareMathOperator{\Id}{Id}

\DeclareMathOperator{\Lip}{Lip}
\DeclareMathOperator{\diam}{diam}

\DeclareMathOperator{\Vol}{Vol}

\DeclareMathOperator{\dist}{dist}

\DeclareMathOperator{\Ric}{Ric}

\DeclareMathOperator{\Div}{div}

\newcommand{\bangle}[1]{\left\langle #1 \right\rangle}

\newcommand{\eps}{\varepsilon}

\begin{document}

\title{Stable anisotropic minimal hypersurfaces in $\mathbf{R}^{4}$}
\author{Otis Chodosh}
\address{Department of Mathematics, Stanford University, Building 380, Stanford, CA 94305, USA}
\email{ochodosh@stanford.edu}
\author{Chao Li}
\address{Courant Institute, New York University, 251 Mercer St, New York, NY 10012, USA}
\email{chaoli@nyu.edu}

\maketitle

\begin{abstract}

We show that a complete, two-sided, stable immersed anisotropic minimal hypersurface in $\RR^4$ has intrinsic cubic volume growth, provided the parametric elliptic integral is $C^2$-close to the area functional. We also obtain an interior volume upper bound for stable anisotropic minimal hypersurfaces in the unit ball. We can estimate the constants explicitly in all of our results.

In particular, this paper gives an alternative proof of our recent stable Bernstein theorem for minimal hypersurfaces in $\RR^4$. The new proof is more closely related to techniques from the study of strictly positive scalar curvature.

\end{abstract}

\section{Introduction}
Consider $\Phi:\RR^{n+1} \setminus\{0\} \to (0,\infty)$ a $1$-homogeneous  $C^{3}_{\textrm{loc}}$ function (i.e., $\Phi(sv) = s\Phi(v)$ for $s>0$). For $M^{n}\to\RR^{n+1}$ a two-sided immersion (with chosen unit normal field $\nu(x)$), we can define the \emph{anisotropic area functional} 
\[\bphi(M) = \int_M \Phi(\nu(x)) \, d\mu.\]
 Surfaces minimizing the $\bphi$-functional arise as the equilibrium shape of crystalline\footnote{We note that in the crystalline setting $\Phi$ is usually only Lipschitz continuous.} materials, as well as scaling limits of Ising and percolation models (see \cite[Chapter 5]{Cerf2006wulff}).

We say that $M$ is \emph{$\bphi$-stationary} if $\tfrac{d}{dt}\big|_{t=0} \bphi(M_{t}) = 0$ for all compactly supported variations of $M$ (fixing $\partial M$) and that $M$ is \emph{$\bphi$-stable} if  in addition $\tfrac{d^{2}}{dt^{2}}\big|_{t=0} \bphi(M_{t}) \geq 0$ for such variations. Note that if $\Phi(v) = |v|$, $\bphi$ reduces to the $n$-dimensional area functional and a $\bphi$-stable hypersurface is known as a \emph{stable minimal hypersurface}. We say that $\bphi$ is \emph{elliptic} if the $\Phi$-unit ball $\Phi^{-1}((0,1]) \cup\{0\}$ is uniformly convex.

This article is motivated by the following questions:
\begin{quest}\label{quest:stable-bern}
For an anisotropic elliptic functional $\bphi$, is the flat hyperplane $\RR^{n}\subset \RR^{n+1}$ the only complete two-sided $\bphi$-stationary and stable immersion in $\RR^{n+1}$?
\end{quest}

\begin{quest}\label{quest:stable-volume-growth}
If $M^{n}\to\RR^{n+1}$ is a complete two-sided $\bphi$-stationary and stable immersion (for some anisotropic elliptic functional $\bphi$) does $M$ satisfy the intrinsic polynomial volume growth condition $\Vol(B_{M}(p,\rho)) \leq C \rho^{n}$?
\end{quest}

By a well-known blowup argument, an affirmative answer to Question \ref{quest:stable-bern} yields a priori interior curvature estimates for $\bphi$-stable immersions with boundary, and even for stable immersion with respect to a \emph{parametric elliptic integrand} (where $\Phi$ is allowed to also depend on $x$). We also note that for minimal surfaces one can derive \emph{lower} polynomial growth bounds (both intrinsic and extrinsic), but for general $\bphi$-stationary surfaces no monotonicity type formula is known, cf.\ \cite{Allard:characterization,DPDRH:area-blowup,DPDRG:rect}. (On the other hand, stability can be used to derive a \emph{lower} volume growth estimate; see Corollary \ref{coro:lower-vol-growth} and \cite{DPDRH:area-blowup}.) 

For the area functional, Question \ref{quest:stable-bern} (and thus Question \ref{quest:stable-volume-growth}) has been completely resolved in the affirmative when $n=2$ (independently) by Fischer-Colbrie--Schoen, do Carmo--Peng, and Pogorelov \cite{Fischer-ColbrieSchoen1980structure,docarmo-peng,pogorelov} (see also \cite{Schoen:estimates}) and recently when $n=3$ by the authors \cite{ChodoshLi:R4}. In particular, we recall the result of Pogorelov (yielding a localized volume growth estimate)
\begin{theo}[{\cite{pogorelov}, cf.\ \cite[Lemma 34]{White:notes}, \cite[Theorem 2]{MunteanuSungWang}}]\label{theo:pog}
Suppose that $M^{2}\to\RR^{3}$ is a stable minimal immersion so that the intrinsic ball $B_{M}(p,R) \subset M$ has compact closure in $M$ and is topologically a disk. Then 
\[
|B_{M}(p,\rho)| \leq \frac{4}{3} \pi \rho^{2}. 
\]
\end{theo}

On the other hand, Questions \ref{quest:stable-bern} and \ref{quest:stable-volume-growth} remains open (even for the area functional) for $n=4,5,6$. There exist non-flat stable minimal hypersurfaces (area minimizers) in $\RR^{8}$ and beyond \cite{BDG,HardtSimon} (thus answering Question \ref{quest:stable-bern} in the negative), but all known examples satisfy the conclusion of Question \ref{quest:stable-volume-growth}. Note that Schoen--Simon--Yau \cite{SSY} (cf. \cite{Simons,SchoenSimon,Wickramasekera}) have shown that when $n\leq 5$, a complete two-sided stable minimal immersion satisfying the volume growth condition in Question \ref{quest:stable-volume-growth} must be flat. 

For arbitrary elliptic functionals, there are non-flat minimizers for $n\geq 3$ \cite{Morgan:cone-cliff,MooneyYang}, but as in the case of area, all known examples satisfy the intrinsic volume growth condition in Question \ref{quest:stable-volume-growth}. When $n=2$, Question \ref{quest:stable-bern} is open for general elliptic functionals but is resolved in the affirmative assuming quadratic area growth (as shown by White \cite{White:existence-parametric}) or assuming the functional is sufficiently $C^{2}$-close\footnote{Throughout this article, \emph{$\Phi$ is $C^{k,\alpha}$-sufficiently close to area} will mean that $\Vert \Phi - 1\Vert_{C^{k,\alpha}(S^{n})} \leq \eps(n)$ for some fixed $\eps(n)>0$. } to area (as shown by Lin \cite{Lin:parametric}; see also \cite{Jenkins,Simon:bern-extend}). Still for $n=2$, Colding--Minicozzi have given a new proof of Theorem \ref{theo:pog} that extends to show that Question \ref{quest:stable-volume-growth} holds in the affimative for functionals sufficiently $C^{2}$-close to area. When $n\geq 3$, Question \ref{quest:stable-bern} is answered in the negative by considering the non-flat area minimizing solutions constructed by Mooney--Yang \cite{MooneyYang} (see also \cite{Morgan:cone-cliff,Mooney2019minimal}). 
On the other hand, Winklmann has resolved Question \ref{quest:stable-bern} in the affirmative for $n\leq 5$ under the assumptions that the functional is sufficiently $C^{4}$-close to area and that the surface satisfies the growth condition from Question \ref{quest:stable-volume-growth}.

\subsection{Main results} 

In this article, we consider the volume growth problem (Question \ref{quest:stable-volume-growth}) for $\bphi$-stable hypersurfaces  $\RR^{4}$. In fact, the estimate we prove here is new even in the case of stable minimal hypersurfaces. As such, it yields an alternative approach to our recent result \cite{ChodoshLi:R4} (this is discussed further in Section \ref{sec:min-hyp}). 

We note that all constants in this paper can be given explicitly, see Remark \ref{rema:constants}. 

\begin{theo}\label{theo:main.global}
Assume that $\Phi$ satisfies 
\begin{equation}\label{assumption.phi}
		|v|^2 \le D^2 \Phi(\nu) (v,v) \le \sqrt{2} |v|^2,
	\end{equation}
for all $v \in \nu^{\perp}$.
	 Consider $M^3\to \RR^4$ a complete, two-sided, $\bphi$-stationary and stable immersion. Suppose $0\in M$, and $M$ is simply\footnote{We note that a standard argument (cf.\ \cite{Fischer-ColbrieSchoen1980structure}) shows that if $M^{3}\to\RR^{4}$ is a complete two-sided $\bphi$-stable immersion, then so is the immersion from the universal cover.} connected. Then there exist explicit constants $V_0=V_0(\|\Phi\|_{C^1(S^{3})}), Q >0$ such that
	\begin{enumerate}[label=(\roman*)]
		\item $|B_{M}(0,\rho)|\le V_0 \rho^3$, for all $\rho>0$.
		\item For each connected component $\Sigma_0$ of $\partial B_{M}(0,\rho)$, we have
		\[\max_{x\in \Sigma_0} r(x)\le Q \min_{x\in \Sigma_0} r(x),\]
		where $r(x):=d_{\RR^4}(0,x)$.
	\end{enumerate}
\end{theo}

Note that \eqref{assumption.phi} implies that $\nu \mapsto \Phi(\nu)$ is convex (since $D^{2}\Phi(\nu)(\nu,\nu) = 0$ by $1$-homogeneity). As such, all $\bphi$ considered in Theorem \ref{theo:main.global} satisfy the ellipticity condition mentioned previously.

We note that by combining Theorem \ref{theo:main.global} with \cite{Wink:pointwise}, we obtain:

\begin{coro}\label{coro:stable-bern-Phi-R4}
If $\Phi$ is $C^{4}$-sufficiently close to area, then any two-sided complete $\bphi$-stationary and stable immersion is flat. 
\end{coro}

\begin{rema}
Although it is not explicitly done in \cite{Wink:pointwise}, the ``sufficiently close'' requirement can be quantified. Alternatively, we note that by combining  Theorem \ref{theo:main.global} with a contradiction argument in the spirit of \cite{Simon:bern-extend}, Corollary \ref{coro:stable-bern-Phi-R4} actually holds under the weaker assumption of $C^{2,\alpha}$-closeness (but with no numerical estimate of the required closeness). 
\end{rema}

We can also prove a localized version of Theorem \ref{theo:main.global} more in the spirit of Pogorelov's result (cf.\ Theorem \ref{theo:pog}). The estimate we prove here is slightly different, since it considers extrinsic balls, but is an interior\footnote{As observed in \cite[\S 1]{GulliverLawson1986singularity}, the bridge principle for stable minimal surfaces \cite{MeeksYau:exist} implies that there cannot be an estimate for the area of a proper stable minimal immersion $M^{2}\to B_{1}(0)\subset \RR^{3}$, even if $M$ is topologically constrained to be a disk. } estimate. Even for stable minimal surfaces, we are not aware of such an estimate in $\RR^{3}$ with explicit\footnote{Given an area-free curvature estimate (available for minimal surfaces when $n=2,3$ \cite{Schoen:estimates,ChodoshLi:R4}), one can prove an extrinsic interior Pogorelov result in the spirit of Theorem \ref{theo:main.local} by a straightforward contradiction argument (with no control on the constant). The method used here gives an alternative proof of this curvature estimate (and extends to certain elliptic integrands) and yields explicit (and not too large) constants. } constants, cf.\ Remark \ref{rema:constants}.  %{\color{blue} I think our local result is actually better, since we can deal with ambient distance? Do you know of an analogous area estimates for surfaces in $\RR^3$?})

\begin{theo}\label{theo:main.local}
	Suppose that $\Phi$ satisfies \eqref{assumption.phi}. Assume $M^3\to B_1(0)\subset \RR^4$ is a proper, two-sided, $\bphi$-stationary and stable immersion. Suppose $0\in M$, $M$ is simply connected, and $\partial M$ is connected. Then there exist explicit constants $\rho_0 \in (0,1), V_1= V_1(\|\Phi\|_{C^1(S^{3})})$, such that
	\[|M_{\rho_0}^*|\le V_1,\]
	where $M_{\rho_0}^*$ is the connected component of $M\cap B_{\RR^{4}}(0,\rho_{0})$ that contains $0$. 
\end{theo}

\begin{rema}
	More generally, we can drop the requirement that $M$ is simply connected and $\partial M$ is connected. In this case, we have:
	\[|M_{\rho_0}^*|\le V_1(b_1(M) + E),\]
	where $E$ is the number of boundary connected components of $M$ and $b_{1}(M)$ is the first Betti number. 
\end{rema}

\begin{rema}\label{rema:constants}
One may explicitly compute the constants $V_0,V_{1}$, $Q$, $\rho_{0}$ as follows. Let 
	\begin{align*}
	c_0 & = \frac{1}{\sqrt 2 - \frac 12},\\
	\lambda & = \frac32 \left(\frac12 - \frac{3(c_0-1)}{8(\frac{1}{\sqrt 2} - \frac12)}\right) = \frac{3(5+3\sqrt 2)}{56} \approx 0.495.
	\end{align*}
	Then we have
	\[V_0 = \frac{8\pi e^{\frac{15\pi}{\lambda}} \|\Phi\|_{C^1(S^3)} }{3\lambda \min_{\nu \in S^3} \Phi(\nu)}, \quad Q = e^{\frac{7\pi}{\sqrt \lambda}},\] 
	and
	\[\rho_0 = e^{-\frac{5\pi}{\sqrt \lambda}},\quad V_1= \frac{8\pi  \|\Phi\|_{C^1(S^3)}}{3\lambda\min_{\nu\in S^3}\Phi(\nu)}.\]
\end{rema}

\subsection{Related work}

We recall here some works (beyond those mentioned above) that are related to this paper. The regularity of hypersurfaces \emph{minimizing} parametric elliptic integrands has been studied in several places including \cite{Federer,SchoenSimonAlmgren1977regularity,SS:par-ell,Figalli:reg-c1alpha}. See also \cite{Allard1983stableelliptic,White1987curvature} for estimates without the minimizing hypothesis. Existence of critical points of parametric elliptic integrands has been considered in \cite{White:space,White:existence-parametric,DPDR:minmax}. Finally, we note that stable solutions for the nonlocal area functional satisfy an a priori growth estimate (as in Question \ref{quest:stable-volume-growth}) in all dimensions \cite{CSV:nonlocal} (see also \cite{FigalliSerra:nonlocalAC}).

\subsection{Notation}
We will use the following notation: 

\begin{itemize}
	\item $B_{\RR^{n+1}}(0,\rho) := \{x\in \RR^{n+1}: |x| < \rho\}$.
	\item $r(x) = \dist_{\RR^{n+1}} (0,x)$.
	\item $M^n\to\RR^{n+1}$ is an immersion and $g$ the induced Riemannian metric on $M$.
	\item $D$ is the connection in $\RR^{n+1}$, $\nabla$ is the induced connection on $M$.
	\item $\mu$ is the volume form of $g$.
	\item $B_{M}(0,\rho) : = \{x\in M: \dist_{M,g} (0,x)<\rho\}$.
	\item $\nu$ is a choice of unit normal vector field of  $M$.
	\item The shape operator will be written $S = \nabla \nu$, and the second fundamental form written $A(X,Y) = S(X)\cdot Y$.
	\item The scalar curvature of $g$ will be denoted by $R$. 
	\item We will use the $\ell^{2}$-norm to define $C^{k}$-norms, i.e.\ $\Vert f \Vert_{C^{k}} : = (\sum_{j=0}^{k} \Vert D^{(j)}f \Vert_{C^{0}}^{2})^{\frac 12}$.
\end{itemize}

\subsection{Organization of the paper}
In Section \ref{sec:min-hyp} we explain the techniques used in this paper in the special case of the area functional. The remaining part of the paper contains the details necessary for the generalization to anisotropic integrands. We begin in Section \ref{sec:prelim-aniso} with some preliminary results. Section \ref{sec:one-end} contains a generalization of the one-ended result for stable minimal hypersurfaces due to Cao--Shen--Zhu to the case of certain anisotropic integrands.  We describe the conformally changed metric in Section \ref{section.conformal} as introduced by Gulliver--Lawson and then combine these techniques with $\mu$-bubbles to prove the main results in Section \ref{section.conformal}.  Appendix \ref{app:var} contains (well-known) computations of the first and second variation for elliptic integrands. Appendix \ref{appendix.quadratic.forms} contains an auxiliary result comparing certain quadratic forms.

\subsection{Acknowledgements}We are grateful to Fang-Hua Lin
 and Guido De Philippis for their interest and for several discussions, as well as Doug Stryker for pointing out a mistake in an earlier version of the paper, and to the referees for some useful suggestions. O.C. was supported by an NSF grant
(DMS-2016403), a Terman Fellowship, and a Sloan Fellowship. C.L. was supported by an NSF grant (DMS-2202343).

\section{Volume growth for stable minimal hypersurfaces in $\RR^{4}$} \label{sec:min-hyp}

In this section, we illustrate how one may use stability to deduce area estimates for stable minimal immersions $M^3\to \RR^4$. We will defer certain ancillary results and computation to later sections (where they were carried out for general $\Phi$-stationary and stable hypersurfaces), and instead focus on the geometric ideas and consequences.

The main result we will prove here is as follows

\begin{theo}\label{theo.minimal}
	Let $M^3\to \RR^4$ be a complete, two-sided, simply connected, stable minimal immersion, $0\in M$. Then,
	\[
	|B_{M}(0,\rho)| \leq \left(\frac{32\pi}{3}\right)^{\frac 32}\frac{e^{\frac{30\pi}{\sqrt 3}}}{6\sqrt \pi} \rho^{3}
	\]
	for all $\rho\geq 0$. 
\end{theo}

Combined with the work of Schoen--Simon--Yau \cite{SSY}, this yields a new proof of our recent result \cite{ChodoshLi:R4}:
\begin{coro}\label{coro:stable-bern-R4}
Any complete, two-sided, stable minimal immersion $M^{3}\to\RR^{4}$ is flat.
\end{coro}

In fact, we have the following localized volume estimate in the spirit of Theorem \ref{theo:pog}.

\begin{theo}\label{theo.minimal-local}
	Let $M^3\to \RR^4$ be a two-sided, simply connected, stable minimal immersion, with $0\in M$, $\partial M$ connected, and $M\to B_{\RR^{4}}(0,1) $ proper. Then,
	\[
	|M_{\rho_{0}}^{*}| \leq \left(\frac{32\pi}{3}\right)^{\frac 32}\frac{1}{6\sqrt \pi}
	\]
	where $M_{\rho_0}^*$ is the connected component of $M\cap B_{\RR^{4}}(0,r_{0})$ that contains $0$ and $\rho_{0} = e^{-\frac{10\pi}{\sqrt 3}}$.
\end{theo}

\begin{proof}[Proof of Theorem \ref{theo.minimal}]
%It follows from \cite{CaoShenZhu1997end} that $M$ has only one end. In particular, $M \setminus B_{M}(0,\rho)$ has exactly one unbounded component for all $\rho>0$.  
The first step is to consider a particular conformal deformation of $(M,g)$. On $M\setminus \{0\}$, consider the conformally deformed metric $\tilde g = r^{-2}g$ (where we recall that $r$ is the \emph{Euclidean} distance to the origin and $g$ is the induced metric on $M$). We use $\tilde \nabla, \tilde \mu, \tilde \Delta$ to denote the covariant derivative, the volume form and the Laplacian with respect to $\tilde g$, respectively. This conformal change was first used by Gulliver--Lawson \cite{GulliverLawson1986singularity} to study isolated singularities for minimal hypersurfaces in $\RR^{n+1}$. 

\begin{rema}
The relevance of the Gulliver--Lawson conformal deformation is a key insight in our work. Indeed, this allows us to apply tools from the study of \emph{strictly} positive scalar curvature (cf.\ Remark \ref{rema:mu-bub}). Our previous proof of Corollary \ref{coro:stable-bern-R4} used tools from \emph{non-negative} scalar curvature (cf.\ \cite{MW:1,MW:2}).\footnote{Added in proof: some time after this paper appeared, Catino--Mastrolia--Roncoroni found a third proof of Corollary \ref{coro:stable-bern-R4}, based on a surprising connection between stability and \emph{non-negative Bakry--\'Emery Ricci curvature} \cite{CMR}.} 
\end{rema}

The computations in this part work for minimal immersions $M^n\to \RR^{n+1}$ whenever $n\ge 3$. For $\lambda\in \RR$, $\varphi\in C_0^1(M\setminus \{0\})$, consider the quadratic form
\[\cQ(\varphi) : =\int_M \left(|\tilde \nabla \varphi|_{\tilde g}^2 + (\tfrac 12 \tilde R - \lambda)\varphi^2\right)d\tilde\mu,\]
where $\tilde R$ is the scalar curvature of $\tilde g$. One computes (see Section \ref{section.conformal} for details) that
\begin{align*}
	\cQ(r^{\frac{n-2}{2}}\varphi) &= \int_M \left(r^2 |\nabla (r^{\frac{n-2}{2}}\varphi)|^2 + (\tfrac12 \tilde R - \lambda )r^{n-2}\varphi^2\right)r^{-n} d\mu\\
	& = \int_M \left(|\nabla \varphi|^2 +\tfrac 12 R \varphi^2 + \left(\frac n2 \left(n  - \frac{n+2}{2}|\nabla r|^2\right )-\lambda\right)r^{-2}\varphi^2\right)d\mu\\
	& \ge \int_M \left(|\nabla \varphi|^2 + \tfrac 12 R \varphi^2 +  \left(\frac{n(n-2)}{4} - \lambda\right)\right) d\mu.
\end{align*}

By the (traced) Gauss equations, minimality of $M$ implies that $|A_{M}|^2 = -R_g$. Thus, we can use stability of $M$ to conclude
\[ \int_M \left( |\nabla \varphi|^2 - |A|^2 \varphi^2\right) d\mu \ge 0 \quad\Rightarrow \quad \int_M \left( |\nabla \varphi|^2 + \frac12 R \varphi^2\right) d\mu \ge 0,\]
for all $\varphi\in C_0^1(M)$. Note that we have used the fact that the scalar curvature of a minimal hypersurface in $\RR^{n+1}$ has $R\le 0$ and that $\tfrac12 <1$. In particular, choosing $\lambda = \frac{n(n-2)}{4}$ above, we find that $\cQ(\varphi)\ge 0$ for any $\varphi \in C_0^1 (M\setminus \{0\})$. Using \cite[Theorem 1]{Fischer-ColbrieSchoen1980structure}, there exists $u\in C^\infty(M\setminus \{0\})$, $u>0$ in the interior of $M\setminus \{0\}$, such that
\begin{equation}\label{eq.spectral.minimal}
	\tilde \Delta u \le -\frac12 \left(\frac{n(n-2)}{2} - \tilde R\right)u.
\end{equation}
We note that \eqref{eq.spectral.minimal} is an integral form of \emph{strictly positive} scalar curvature. 

In the second step, we restrict to the case of $n=3$. We use warped $\mu$-bubbles to derive geometric inequalities for $3$-manifolds $(N^3,g)$ admitting a positive function $u$ with \eqref{eq.spectral.minimal}. 
\begin{rema}\label{rema:mu-bub}
The $\mu$-bubble technique was first used by Gromov \cite[Section 5$\tfrac56$]{Gromov1996positive} (see also \cite{Gromov2018metric}). Warped $\mu$-bubbles have previously been combined with minimal hypersurface techniques to study problems in scalar curvature and in minimal surfaces, see, e.g. \cite{ChodoshLi2020bubble,Gromov2020aspherical,ChodoshLiLiokumovich2021,ChodoshLiStryker2022stableminimal,Zhu:width,Zhu:rigidity}. Precisely, suppose $n=3$ and $\partial N\ne\emptyset$. Then there exists an open set $\Omega$ containing $\partial N$, $\Omega\subset B_{\frac{10\pi}{\sqrt 3}}(\partial N)$, such that each connected component of $\partial \Omega\setminus \partial N$ is a $2$-sphere with area at most $\frac{32\pi}{3}$ and intrinsic diameter at most $\frac{4\pi}{\sqrt 3}$ (see Lemma \ref{lemm:mu.bubble}). 
\end{rema}
Fix $\rho>0$. By \cite{CaoShenZhu1997end}, $M\setminus B_M\left(0,e^{\frac{10\pi}{\sqrt 3}}\rho \right)$ has only one unbounded component $E$. Denote by $M'=M\setminus E$. We apply Remark \ref{rema:mu-bub} to $N=M'$, and find $M_0\subset M'$ with $\dist_{ \tilde g} (\partial M_0, \partial M')\le \frac{10\pi}{\sqrt 3}$. The topological assumptions on $M$ force $\partial M_0$ to be connected, so $|\partial M_0|_{\tilde g}\le \tfrac{32\pi}{3}$ and $\partial M_{0}$ has intrinsic diameter $\leq \frac{4\pi}{\sqrt{3}}$. By comparing $g$-distance with $\tilde g$-distance (see (2) in Lemma \ref{lemm.compare.distance}), we find that 
\[
B_{M}(0,\rho) \subset M_{0} \subset B_{M}(0,e^{\frac{10\pi}{\sqrt{3}}} \rho). 
\]
In particular, bounding intrinsic distance by extrinsic distance, we see that $\sup_{\partial M_{0}} r(x) \leq e^{\frac{10\pi}{\sqrt{3}}} \rho$. Thus, we have
\[
|B_{M}(0,\rho)| \le | M_0| \le \frac{1}{6\sqrt \pi} |\partial M_0|_g^{\frac 32}  \le \frac{1}{6\sqrt \pi} (e^{\frac{10\pi}{\sqrt{3}}}\rho)^3 |\partial M_0|_{\tilde g}^{\frac 32} \le \left(\frac{32\pi}{3}\right)^{\frac 32}\frac{e^{\frac{30\pi}{\sqrt 3}}}{6\sqrt \pi}\rho^3,
\]
where in the second step we have used the isoperimetric inequality for minimal hypersurfaces in Euclidean spaces due to  Brendle \cite{Brendle2021isoperimetric} (cf.\ \cite{MS:sob}). This completes the proof. 
\end{proof}

We now consider the requisite changes needed to prove the local result:
\begin{proof}[Proof of Theorem \ref{theo.minimal-local}]
In the case where $M$ is properly immersed in $B_1(0)\subset\RR^4$, we proceed similarly as before, and obtain a region $M'$ such that $\dist_{ \tilde g} (\partial M',\partial B_1(0))\le \frac{10\pi}{\sqrt{3}}$,  $|\partial M'|_{\tilde g}\le \frac{32\pi}{3}$, and $\partial M'$ is connected. Again, using Lemma \ref{lemm.compare.distance}, we conclude that
\[M_{\rho_0}^*\subset M',\]
where $\rho_0 = e^{-\frac{10\pi}{\sqrt 3}}$, and $M_{\rho_0}^*$ is the connected component of $M\cap B_{\RR^{4}}(0,\rho_{0})$ that contains $0$. Using \cite{Brendle2021isoperimetric} as above, 
\[|M_{r_0}^*| \le |M'| \le \frac{1}{6\sqrt \pi} |\partial M'|_g^{\frac 32}\le \frac{1}{6\sqrt \pi}  |\partial M'|_{\tilde g}^{\frac 32} \le \left(\frac{32\pi}{3}\right)^{\frac 32}\frac{1}{6\sqrt \pi}.\]
This completes the proof. 
\end{proof}

\section{Preliminaries on anisotropic integrands} \label{sec:prelim-aniso}

We now consider a general anisotropic elliptic integrand. For $M^{n}\to\RR^{n+1}$ two-sided immersion, we can set
\[
\bphi(M) = \int_{M}\Phi(\nu(x))\, d\mu. 
\]
In this section we discuss the first and second variation formulae, as well as some important consequences to be used later. 
\subsection{First variation}
Recall that $M$ is $\bphi$-stationary means that $\tfrac{d}{dt}\big|_{t=0} \bphi(M_{t}) = 0$ for all compactly supported variations $M_t$ fixing $\partial M$. By \eqref{eq:first-var}, \eqref{eq:HPhi}, \eqref{eq:Hphi-trace} this is equivalent to 
\[
\Div_{M}(D\Phi(\nu)) = \tr_{M}(\Psi(\nu) S_{M}) = 0 ,
\]
which we can interpret as vanishing of the $\bphi$-mean curvature. Here, $\Psi(\nu) : T\RR^{n+1}\to T\RR^{n+1}$ is defined by $\Psi(\nu): X \mapsto D^{2}\Phi(\nu)[X,\cdot]$ and $S_{M}$ is the shape operator of $M$. 

By the calculation in Section \ref{subsec:first-var-vf}, we find that if $M$ is $\bphi$-stationary, then for any compactly supported  (but not necessarily normal) vector field $X$ along $\Sigma$, we have
\begin{equation}\label{eq.1st.variation.vf}
\int_M \Phi(\nu)\Div_M X + D_{D\Phi(\nu)^T}X\cdot \nu = \int_{\partial M} \Phi(\nu) X\cdot \eta + (X\cdot\nu) D\Phi(\nu)\cdot \eta. 
\end{equation}

By plugging the position vector field into \eqref{eq.1st.variation.vf}, we obtain the following isoperimetric type inequality. 

\begin{coro}\label{coro.isop.ineq}
	Suppose $M^n\to \RR^{n+1}$ is $\bphi$-stationary, and the image of $\partial M$ is contained in $B_{\RR^{n+1}}(0,\rho)$ for some $\rho>0$. Then
	\[|M| \le \frac{\rho\|\Phi\|_{C^1(S^{n})}}{n \cdot \min_{\nu\in S^n} \Phi(\nu)} |\partial M|.\]
\end{coro}

\begin{proof}
	Recall that $r(x) = \dist_{\RR^{n+1}}(x,0)$. Plug $X = \sum_{i=1}^{n+1} x_i e_i$, the position vector field in $\RR^{n+1}$, into \eqref{eq.1st.variation.vf}. Then $\Div_M X = n$, and
	\[D_{D\Phi(\nu)^T} X \cdot \nu = \sum_i (D_{D\Phi(\nu)^T} x_i) (e_i\cdot \nu) = \sum_i (D\Phi(\nu)^T \cdot e_i) (e_i\cdot \nu) = D\Phi(\nu)^T \cdot \nu = 0.\]
	On the other hand, $|X(x)|\le r(x)$. Thus, we find (using $\nu,\eta$ orthonormal) 
	\[\int_M n \Phi(\nu) \le \int_{\partial M} \|\Phi\|_{C^1(S^n)} |X|\le \rho\|\Phi\|_{C^1(S^n)} |\partial M|.\]
	This completes the proof. 
\end{proof}

The next lemma generalizes the traced Gauss equation $R = - |A|^{2}$ (valid for minimal hypersurfaces) to the case of $\bphi$-stationary hypersurfaces in $\RR^{4}$, under the assumption that $D^{2}\Phi(\nu)$ is sufficiently pinched. 
\begin{lemm}\label{lemm:compare.A.and.R}
	Suppose $\Phi$ satisfies \eqref{assumption.phi} and $M^{3}\to\RR^{4}$ is $\bphi$-stationary. Then at each point on $M$, the induced scalar curvature satisfies $R \le 0$ and 
	\begin{equation}\label{eq.compare.A.and.R}
		-R \le |A|^2 \le -c_0R,
	\end{equation} 
	where 
	\[c_0 = \frac{1}{\sqrt 2 - \frac 12} \approx 1.09 \]
\end{lemm}

\begin{proof}
Recall that $\bphi$-stationarity can be written as $\tr_{M}(\Psi(\nu)S_{M}) = 0$. Diagonalizing $A_{M}$ at a given point, write $k_{i}$ for the principal curvatures of $M$ and $e_{i}$ for corresponding principal directions. Thus, $\bphi$-stationarity can be written as
\[
0 = \sum_{i=1}^{3} a_{i}k_{i}
\] 
where $a_{i} = D^{2}\Phi(\nu)[e_{i},e_{i}]$. Without loss of generality, we can assume that $a_1\le a_2\le a_3$. Note that the pinching assumption \eqref{assumption.phi} yields
\[
1 \leq a_1\le a_2\le a_3 \leq \sqrt{2}. 
\]
We have $|A|^2 = \sum k_i^2$, $R = 2\sum_{i< j} k_ik_j$. Writing $k_3= -\frac{a_1k_1 + a_2k_2}{a_3}$, we have
	\[|A|^2 = Q_1(k_1,k_2) :=\frac{a_1^2 +a_3^2}{a_3^2} k_1^2 + \frac{2a_1a_2}{a_3^2}k_1k_2 + \frac{a_2^2+a_3^2}{a_3^2} k_2^2,\]
	\[-R = Q_2(k_1,k_2):= \frac{2a_1}{a_3} k_1^2 + \frac{2(a_1+a_2-a_3)}{a_3}k_1k_2 + \frac{2a_2}{a_3} k_2^2.\]
	By the Gauss equation, we have $R + |A|^2 = H^2 \ge 0$, and hence $|A|^2 \ge -R$. Moreover, whenever $(a_1+a_2-a_3)^2<4a_1a_2$ (which is guaranteed by, for instance, $a_3<4a_1$), $Q_2$ is a positive definite quadratic form, and hence $-R$ is nonnegative. Given that $\frac{a_3}{a_1}, \frac{a_3}{a_2}\in [1,\sqrt{2}]$, \eqref{eq.compare.A.and.R} follows from Appendix \ref{appendix.quadratic.forms}. 
\end{proof}

\subsection{Second variation} 
Suppose now that $M^{n}\to\RR^{n+1}$ is $\bphi$-stationary and stable. 
In Section \ref{subsec:2nd-var} we derive the following second variation formula. 
\begin{equation}\label{eq.second.variation}
\frac{d^2}{dt^2}\bigg\vert_{t=0}\bphi(M_t) = \int_M \bangle{\nabla u, \Psi(\nu) \nabla u} - \tr_{M} \left(\Psi(\nu) S_M^2\right) u^2,
\end{equation}
where $u\nu$ is the variation vector field. Note that stability and \eqref{eq.second.variation} implies that
\begin{equation}\label{eq.almost.stability.general}
\int_M |\nabla u|^2 - \Lambda |A|^2 u^2 \ge 0
\end{equation}
for all $u\in C_c^1(M\setminus\partial M)$. Here, $\Lambda$ depends on ellipticity of $\bphi$. It is important to observe that if $\Phi$ satisfies \eqref{assumption.phi} then $\Lambda \geq \tfrac{1}{\sqrt{2}}$ and in particular%. In particular, we have that a $\bphi$-stable hypersurface (assuming that $\Phi$ satisfies \eqref{assumption.phi}) has  %then \eqref{eq.second.variation} implies that for $\bphi$-stable surfaces $M$, we have
\begin{equation}\label{eq.almost.stability.root2}
\int_M |\nabla u|^2 - \frac{1}{\sqrt 2} |A|^2 u^2\ge 0
\end{equation}
for all $u\in C_c^1(M\setminus\partial M)$.

\subsection{Sobolev inequality and its consequences}\label{section.sobolev}
In this section, we assume that $n\ge 3$, $M^n$ is a two-sided $\bphi$-stationary and stable hypersurface immersed in $\RR^{n+1}$, where $\bphi$ is a general anisotropic elliptic integral. The Michael-Simon Sobolev inequality \cite{MS:sob} implies that for any $f\in C_c^1(M)$, 
\[C_n \left(\int_M |f|^{\frac{n}{n-1}} \right)^{\frac{n-1}{n}}\le \int_{M } |\nabla f| + |fH|.\]
(See also \cite{Brendle2021isoperimetric}.) 

Replacing $f$ by $f^{\frac{2(n-1)}{n-2}}$, we find:
\begin{equation}\label{eq.ms.sobolev}
	C_n \left(\int_ M |f|^{\frac{2n}{n-2}}\right)^{\frac{n-1}{n}}\le \int_ M \frac{2(n-1)}{n-2} |f|^{\frac{n}{n-2}} |\nabla f| + |f|^{\frac{2(n-1)}{n-2}} |H|.
\end{equation}
By the H\"older inequality,
\[\int_ M  |f|^{\frac{2(n-1)}{n-2}} |H|  \le \left(\int_M f^2 H^2 \right)^{\frac 12} \left( \int_M |f|^{\frac{2n}{n-2}} \right)^{\frac 12}  \]
The $\Phi$-stability inequality implies
\[\int_ M f^2 H^2 \le n\int_M f^2 |A|^2 \le C(\Phi) \int_{M } |\nabla f|^2.\]
Now we use the H\"older inequality on the first term of the right hand of \eqref{eq.ms.sobolev} and conclude the following Sobolev inequality:
\begin{equation}\label{eq.Sobolev}
	\left(\int_{M } |f|^{\frac{2n}{n-2}}\right)^{\frac{n-2}{n}}\le C(n,\Phi) \int_{M } |\nabla f|^2.
\end{equation}

\begin{coro}\label{coro:lower-vol-growth}
	Suppose $M^n\to\RR^{n+1}$ is $\bphi$-stationary and stable. Assume that $B_{M}(p,\rho)\subset M$ has compact closure. Then,
\[
|B_{M}(p,\rho/2)| \ge  C(n,\Phi) \rho^n.
\]
\end{coro}

\begin{proof}
	For any $u\in C^1(M)$ such that $u\ge 0$ and $\Delta u\ge 0$, the Sobolev inequality \eqref{eq.Sobolev} and the standard Moser iteration implies that, for any $\theta\in (0,1)$, $s>0$, 
	\[\sup_{B_{M}(p,\theta\rho)} u \le  C(n,\theta,\Phi, s)\left(\rho^{-n} \int_{B_{M}(p,\rho)} u^s\right)^{1/s}.\]
	The result follows by taking $u=1$, $s=1$ and $\theta= \tfrac12$.
\end{proof}

\begin{rema}
	The use of Sobolev inequality for volume lower bound was first used by Allard \cite[7.5]{Allard1972first}.
\end{rema}

\begin{coro}\label{coro.infinite.volume}
	Suppose $M^n \to \RR^{n+1}$ is two-sided complete, $\bphi$-stationary and stable, and $K$ is a compact subset of $M$. Then each unbounded component of $M\setminus K$ has infinte volume.
\end{coro}

\begin{proof}
Let $E$ be an unbounded component of $M\setminus K$. Suppose the contrary, that $|E|<V<\infty$. Choose $\rho$ such that $C(n,\Phi) \rho^n >V$. By completeness, there exists $p\in E$ such that $d_M (p,\partial E)>\rho$. Then we have
	\[V >|E| > |B_{M}(p,\rho)|> C(n,\Phi)\rho^n > V, \]
	a contradiction. This completes the proof. 
\end{proof}

Combining \eqref{eq.Sobolev} and Corollary \ref{coro.infinite.volume}, the same argument as used by Cao--Shen--Zhu \cite{CaoShenZhu1997end} implies the following result: %that, if a complete $\Phi$-stationary and stable immersion $M$ has more than one ends, then it admits a harmonic function $u$ with finite Dirichlet energy.

\begin{coro}\label{coro:CSZ-harm-ends}
If $M^{n}\to\RR^{n+1}$ is complete, two-sided $\bphi$-stationary and stable immersion with at least two-ends, then there is a bounded non-constant harmonic function on $M$ with finite Dirichlet energy. 
\end{coro}

\section{One-endedness}\label{sec:one-end}

Through this section we assume that $n=3$, $M^3\to \RR^4$ is $\bphi$-stationary and stable. By analyzing harmonic functions on $M$, we will show that $M$ has only one end, if $\bphi$ satisfies \eqref{assumption.phi} (following \cite{SY:harmonic.stable.minimal,CaoShenZhu1997end}). 

\begin{lemm}
	Suppose that $M^3$ is a complete, two-sided, $\bphi$-stationary and stable immersion in $\RR^4$, and $u$ is a harmonic function on $M$. Then
	\begin{equation}\label{eq.sy.harmonic}
		(\Lambda - \tfrac{1}{\sqrt 2}) \int_M \varphi^2 |A|^2 |\nabla u|^2 + \frac12 \int_M \varphi^2 |\nabla |\nabla u||^2 \le \int_M |\nabla \varphi| ^2 |\nabla u|^2, 
	\end{equation}
	for any $\varphi\in C_0^1(M)$. Here $\Lambda = \Lambda(\Phi)$ is the constant in \eqref{eq.almost.stability.general}.
\end{lemm}

\begin{proof}
Fix $p\in M$. Let $k_i$ be the principal curvatures, $e_i$ be the corresponding orthonormal principal directions diagonalizing $A_{M}$. 

We first show that for any immersed hypersurface $M^3$ in $\RR^4$, equipped with the induced metric, $p\in M$, and any unit vector $v\in T_p M$, we have
\[\Ric (v,v)\ge -\frac{1}{\sqrt 2} |A|^2.\]
Write $v=\sum y_i e_i$. Then $\sum y_i^2=1$. By the Gauss equation, we have
\[\Ric(e_i,e_j) = \sum_k \textrm{Rm}(e_i,e_k,e_k,e_j) = \sum_k \left(A(e_k,e_k)A(e_i,e_j) - A(e_i,e_k)A(e_j,e_k)\right),\]
and thus $\Ric(e_i,e_j) = 0$ when $i\ne j$, and $\Ric(e_i,e_i)= \sum_{j\ne i} A(e_i,e_i)A(e_j,e_j)$. Therefore,
\[\Ric(v,v) = \sum_i \sum_{j\ne i} A(e_j,e_j)A(e_i,e_i) y_i^2 = k_1(k_2+k_3)y_1^2 + k_2(k_3+k_1)y_2^2 + k_3(k_1+k_2)y_3^2. \]
By Cauchy-Schwarz and the AM-GM inequality,
\begin{multline*}
	k_1^2+k_2^2+k_3^2 \ge k_1^2 + \frac12 (k_2+k_3)^2 \ge -\sqrt 2 k_1(k_2+k_3) \\\Rightarrow \quad k_1(k_2+k_3)
	\ge -\frac{1}{\sqrt 2} \sum_i k_i^2 = -\frac{1}{\sqrt 2} |A|^2. 
\end{multline*}
Similarly,
\[k_2(k_3+k_1) \ge -\frac{1}{\sqrt 2} |A|^2, \quad k_3(k_1+k_2)\ge -\frac{1}{\sqrt 2} |A|^2.\]
Therefore, 
\begin{equation} \label{eq.compare.ricci.A}
\Ric(v,v)\ge -\frac{1}{\sqrt 2}|A|^2 \sum_i y_i^2 = -\frac{1}{\sqrt 2}|A|^2.
\end{equation}
Applying this to $\nabla u$, we conclude that:
\[\Ric(\nabla u, \nabla u)\ge -\frac{1}{\sqrt 2}|A_M|^2 |\nabla u|^2.\]

Since $M$ is $\Phi$-stable, \eqref{eq.almost.stability.general} yields
\[\int_M \Lambda |A|^2 \varphi^2 \le \int_M |\nabla \varphi|^2, \quad\forall \varphi\in C_0^1(M).\]
Replacing $\varphi$ by $|\nabla u|\varphi$, we have:
\begin{equation}\label{eq.sy.1}
	\begin{split}
		\int_M \varphi^2 |\nabla u|^2 |A|^2 & \le \int_M |\nabla \varphi|^2 |\nabla u|^2 + 2\int_M \left(\varphi |\nabla u| \bangle{\nabla \varphi, \nabla |\nabla u|} +  \varphi^2 |\nabla |\nabla u||^2\right)\\
		& = \int_M |\nabla \varphi|^2 |\nabla u|^2 - \int_M \varphi^2 |\nabla u| \Delta|\nabla u|.
	\end{split}
\end{equation}
By the improved Kato inequality,
\[|\nabla^{2} u|^2 \ge \frac38 |\nabla u|^{-2} |\nabla |\nabla u|^2|^2. \]
Combined with the Bochner formula and \eqref{eq.compare.ricci.A}, we have:
\begin{equation}
	\begin{split}
		\Delta |\nabla u|^2 &= 2\Ric_M (\nabla u,\nabla u) + 2 |\nabla^{2} u|^2\\
		&\ge -\sqrt 2 |A|^2 |\nabla u|^2 + \frac34 |\nabla u|^{-2} |\nabla |\nabla u|^2|^2.
	\end{split}
\end{equation}
Thus, 
\begin{equation}\label{eq.sy.2}
	\Delta |\nabla u|\ge -\frac{1}{\sqrt 2} |A|^2 |\nabla u| + \frac12 	|\nabla u|^{-1} |\nabla |\nabla u||^2. 
\end{equation}
\eqref{eq.sy.harmonic} follows from \eqref{eq.sy.1} and \eqref{eq.sy.2}.

\end{proof}

\begin{prop}\label{prop.one.ended}
	Suppose $\Phi$ satisfies \eqref{assumption.phi}. Then any complete, two-sided, $\Phi$-stable immersion $M^3$ in $\RR^4$ has only one end.
\end{prop}

\begin{proof}
	Suppose the contrary, that $M$ has at least two ends. Then Corollary \ref{coro:CSZ-harm-ends} implies that $M$ admits a nontrivial harmonic function $u$ with $\int_M |\nabla u|^2 \le C< \infty$. For $\rho>0$, take $\varphi\in C_c^1(M)$ such that $\varphi|_{B_M(0,\rho)}=1$, $\varphi|_{B_{M}(0,2\rho)}=0$, and $|\nabla \varphi|\le \tfrac 2\rho$. Then \eqref{eq.sy.harmonic} implies that
	\begin{equation*}
			\int_{B_{M}(0,\rho)} (\Lambda - \tfrac{1}{\sqrt 2}) |A|^2 |\nabla u|^2 + \frac 12 |\nabla |\nabla u||^2 \le \frac{4}{\rho^2}\int_M |\nabla u|^2\le \frac{4C}{\rho^2}.
	\end{equation*}
	Here $\Lambda \ge \tfrac{1}{\sqrt 2}$ by \eqref{assumption.phi}. Sending $\rho\to \infty$, we conclude that 
	\[ |\nabla |\nabla u||^2\equiv 0. \]
	In particular, this implies that $|\nabla u|$ is a constant. Since $u$ is nonconstant, we have that $|\nabla u| >0$. However, this implies that
	\[\int_{M} 1 = \frac{1}{|\nabla u|^2} \int_M |\nabla u|^2 <\infty,\]
	contradicting Corollary \ref{coro.infinite.volume}.
\end{proof}

\section{A conformal deformation of metrics}\label{section.conformal}

Take $M^3\to \RR^4$ to be $\bphi$-stable, where $\Phi$ satisfies \eqref{assumption.phi}. In this section we carry out the conformal deformation technique used by Gulliver-Lawson \cite{GulliverLawson1986singularity} on $M$.

Consider the function $r(x)=\dist_{\RR^{n+1}}(0,x)$ on $M$, and the position vector field $\vec X$. Then $\Delta \vec X = \vec H$. Thus, $\Delta (r^2) = \Delta (\sum x_i^2 )= 2\vec X\cdot \Delta X + 2|\nabla \vec X|^2 = 2 \vec X\cdot \vec H + 2n$. We find:
\[\Delta r = \frac{n}{r} + H(\hat x\cdot \nu) - \frac{|\nabla r|^2}{r},\]
here $\hat x= \frac{\vec X}{|\vec X|}$ is the normalized position vector.

Suppose that $w>0$ is a smooth function on $M^n\setminus \{0\}$. On $M\setminus \{0\}$, define $\tilde g = w^2 g$. For $\lambda \in \RR$, $\varphi\in C_c^1(M\setminus \{0\})$ consider the quadratic form
\[\cQ_w(\varphi) = \int_M \left(|\tilde \nabla \varphi|_{\tilde g}^2 + (\tfrac12 \tilde R - \lambda)\varphi^2 \right)d\tilde \mu,\]
where $\tilde \nabla$, $\tilde R$, $\tilde \mu$ are the gradient, the scalar curvature and the volume form with respect to $\tilde g$, respectively. One relates the geometric quantities in $g$ and $\tilde g$ as follows:
\[|\nabla \varphi|_g^2 = w^2 |\tilde \nabla \varphi|_{\tilde g}^2 ,\quad d\mu = w^{-n} d\tilde \mu.\]
Moreover, we have
\[w^2 \tilde R = R - 2(n-1)\Delta \log w - (n-1)(n-2) |\nabla \log w|^2.\]

Denote by $\tilde \cQ_w (\varphi) := \cQ_w(w^{\frac{2-n}{2}}\varphi)$. We compute:
\begin{equation*}
	\begin{split}
		&\tilde \cQ_w(\varphi)\\
		&=\int_M\left(w^{-2} |\nabla (w^{\frac{2-n}{2}}\varphi)|_{g}^2 + (\tfrac12 \tilde R - \lambda)w^{2-n} \varphi^2 \right)w^n d\mu\\
		&=\int_M \left(w^{n-2} |w^{\frac{2-n}{2}} \nabla \varphi - \tfrac{n-2}{2} \varphi w^{-\frac n2} \nabla w|_{g}^2 + (\tfrac12w^2 \tilde R  - w^2\lambda)\varphi^2  \right)d\mu\\
		&=\int_M \left( |\nabla \varphi - \tfrac{n-2}{2} \varphi \nabla \log w|_{g}^2 + (\tfrac12 w^2 \tilde R - w^2 \lambda)\varphi^2 \right)d\mu\\
		&=\int_M \left(|\nabla \varphi|_{g}^2 - \frac{n-2}{2} \bangle{\nabla (\varphi^2) , \nabla \log w }_g +\frac{(n-2)^2}{4} |\nabla \log w|_{g}^2\varphi^2 + (\tfrac12 w^2 \tilde R - w^2 \lambda)\varphi^2  \right)d\mu\\
		&= \int_M \left( |\nabla \varphi|_{g}^2 + \left(\frac{n-2}{2} \Delta \log w + \frac{(n-2)^2}{4}|\nabla \log w|_{g}^2  + \tfrac12 w^2 \tilde R - w^2\lambda\right)\varphi^2\right)d\mu\\
		&= \int_M \left(|\nabla \varphi|_{g}^2 + \tfrac12 R\varphi^2  - \left(\frac n2 \left(\Delta \log w + \tfrac{(n-2)}{2} |\nabla \log w|_{g}^2\right) + w^2\lambda  \right)\varphi^2  \right)d\mu.
	\end{split}
\end{equation*}
We now choose $w= r^{-1}$ on $M\setminus \{0\}$. Note that (dropping the $g$ subscript on the norm of the gradient)
\begin{equation*}
	\begin{split}
		\Delta \log	 w + \frac{n-2}{2} |\nabla \log w|^2 &= -\frac{\Delta r}{r} + \frac n2 \frac{|\nabla r|^2}{r^2}\\
														 &= -\frac {n}{r^2} - \frac{H(\hat x\cdot \nu)}{r} + \frac{n+2}{2} \frac{|\nabla r|^2}{r^2}.
	\end{split}
\end{equation*}
Therefore,
\begin{equation}
	\begin{split}
		&\tilde \cQ_w(\varphi)= \int_M \left(|\nabla \varphi|^2 + \tfrac12 R\varphi^2 + \left(\frac n2 \left(n + rH(\hat x\cdot \nu) - \tfrac{n+2}{2} |\nabla r|^2  \right) - \lambda \right)r^{-2}\varphi^2 \right)d\mu\\
		&\ge \int_M \left(|\nabla \varphi|^2 + \tfrac12 R\varphi^2 +\left(\frac n2 \left( n - \tfrac12 \beta r^2 H^2 - \tfrac{1}{2\beta} - \tfrac{n+2}{2} |\nabla r|^2\right) -\lambda \right)r^{-2}\varphi^2  \right)d\mu\\
		&=\int_M \left(|\nabla \varphi|^2 + (\tfrac12 R - \tfrac n4 \beta H^2)\varphi^2 + \left(\frac n2 \left(n - \tfrac{1}{2\beta} - \tfrac{n+2}{2} |\nabla r|^2\right) -\lambda\right) r^{-2}\varphi^2\right)d\mu,
	\end{split}
\end{equation}
for $\beta>0$ to be chosen later.

By the Gauss equation and Lemma \ref{lemm:compare.A.and.R},
\[H^2 = |A|^2 + R \le (1-c_0)R.  \]
Combining with $|\nabla r|\le 1$, we have
\begin{equation}
	\tilde\cQ_w (\varphi) \ge \int_M \left(|\nabla \varphi|^2 + (\tfrac 12 +\tfrac n4 \beta (c_0-1))R \varphi^2 +\left(\frac n2 (\tfrac{n-2}{2}-\tfrac{1}{2\beta})- \lambda \right)r^{-2}\varphi^2\right) d\mu.
\end{equation}
On the other hand, \eqref{eq.almost.stability.root2} and \eqref{eq.compare.A.and.R} imply that for every $\varphi\in C_c^1(M)$,
\[ \int_M \left( |\nabla \varphi|^2 + \frac{1}{\sqrt 2} R\varphi^2\right)d\mu\ge 0. \]
Note that $R\le 0$. Thus, by choosing
\[\beta = \frac{4(\frac{1}{\sqrt 2} - \frac12 )}{n(c_0 -1)}, \quad \lambda = \frac n2 \left(\frac{n-2}{2}- \frac{1}{2\beta}\right) = \frac n2 \left(\frac{n-2}{2} - \frac{n(c_0-1)}{8(\frac{1}{\sqrt 2} - \frac 12)}\right),\]
we have that $\tilde \cQ_w(\varphi)\ge 0 $ for all $\varphi\in C_c^1(M\setminus \{0\})$. We summarize these in the following Proposition.

\begin{prop}\label{prop.conformal.trick}
	Suppose $n\ge 3$, $(M^n,g)$ is an immersed hypersurface in $\RR^{n+1}$, $\Lambda, c_0 \in \RR$, such that:
	\[\int_M (|\nabla \varphi|^2 + \Lambda R \varphi^2 ) dV_M \ge 0,\quad \forall \varphi\in C_c^1(M),\]
	\[\Lambda > \tfrac 12, \quad c_0\ge 1, \quad |A|^2 \le -c_0 R_M.\]
	Then the conformally deformed manifold $(M\setminus \{0\}, \tilde g = r^{-1} g)$ satisfies
	\[\lambda_1 (-\tilde \Delta + \tfrac 12 \tilde R) \ge \lambda,\]
	where $\lambda = \frac n2 \left(\frac{n-2}{2} - \frac{n(c_0-1)}{8(\Lambda - \frac 12)}\right)$.
\end{prop}

\section{Volume estimates}\label{sec:vol}

We first recall a diameter bound for warped $\mu$-bubbles in $3$-manifolds satisfying $\lambda_1(-\Delta + \tfrac12 R)\ge \lambda >0$.

\begin{lemm}[Warped $\mu$-bubble area and diameter bound]\label{lemm:mu.bubble}
	Let $(N^3,g)$ be a $3$-manifold with compact connected boundary satisfying
	\begin{equation}\label{eq.lemm.spec}
		\lambda_1(-\Delta + \tfrac12 R) \ge \lambda >0.
	\end{equation}
	Suppose there exists $p\in N$ such that $d_N(p,\partial N)\ge \frac{5\pi}{\sqrt \lambda}$. Then there exists a connected open set $\Omega$ containing $\partial N$, $\Omega\subset B_{\frac{5\pi}{\sqrt\lambda}} (\partial N)$, such that each connected component of $\partial \Omega\setminus \partial N$ is a $2$-sphere with area at most $\tfrac{8\pi}{\lambda}$ and intrinsic diameter at most $\frac{2\pi}{\sqrt\lambda}$.
\end{lemm}

\begin{proof}
	This is an application of estimates for the warped $\mu$-bubbles (see, e.g. \cite[Section 3]{ChodoshLi2020bubble}). Since $N$ satisfies \eqref{eq.lemm.spec}, there exists $u\in C^\infty(N)$, $u>0$ in $\mathring{N}$, such that
	\begin{equation}\label{eq.lemm.spec2}
		\Delta_N u \le -\tfrac12 (2\lambda - R_N) u.
	\end{equation}
	Take $\varphi_0\in C^\infty(M)$ to be a smoothing of $d_N(\cdot, \partial N)$ such that $|\Lip(\varphi_0)|\le 2$, and $\varphi_0 = 0$ on $\partial N$. Choose $\eps \in (0,\tfrac 12)$ such that $\eps, \frac{4}{\sqrt \lambda}\pi +2\eps$ are regular values of $\varphi_0$.  Define
	\[\varphi = \frac{\varphi_0 - \eps}{\frac{4}{\sqrt \lambda}+ \frac{\eps}{\pi}} - \frac{\pi}{2},\]
	$\Omega_1 = \{x\in N: -\tfrac{\pi}{2}<\varphi <\tfrac{\pi}{2}\}$, and $\Omega_0 = \{x\in N: -\tfrac{\pi}{2}<\varphi\le 0 \}$. We have that $|\Lip(\varphi)|<\tfrac{\sqrt{\lambda}}{2}$. In $\Omega_1$, define $h(x)= -\tfrac12 \tan(\varphi(x))$. By a direct computation, we have \begin{equation}\label{eq.lemm.h}
		\lambda + h^2 - 2|\nabla h|\ge 0.
	\end{equation}
	Minimize 
	\[\cA(\Omega) = \int_{\partial \Omega} u d\cH^2 - \int_{\Omega_1} (\chi_\Omega - \chi_{\Omega_0}) hu d\cH^3,\]
	among Caccioppoli sets $\Omega$ in $\Omega_1$ with $\Omega\Delta \Omega_0$ is compactly contained in $\Omega_1$. By \cite[Proposition 12]{ChodoshLi2020bubble}, a minimizer $\tilde\Omega$ exists and has regular boundary. We take $\Omega$ to be the connected component of $\{x\in N: 0\le \varphi_0(x)\le \eps \}\cup \tilde \Omega$ that contains $\partial N$ (in other words, we disregard any component of $\tilde \Omega$ that is disjoint from $\partial N$). We verify that $\Omega$ satisfies the conclusions of Lemma \ref{lemm:mu.bubble}. 
	
	Indeed, for any connected component $\Sigma$ of $\partial \Omega\cap \Omega_1$, the stability of $\cA$ implies \cite[Lemma 14]{ChodoshLi2020bubble}:
	\begin{multline}
		\int_\Sigma |\nabla \psi|^2 u  - \tfrac12 (R_N - \lambda - 2K_\Sigma) \psi^2 u + (\Delta_N u - \Delta_\Sigma u)\psi^2 \\
			-\tfrac 12 u^{-1} \bangle{\nabla_N u ,\nu}^2 \psi^2 - \tfrac12 (\lambda +h^2 +2\bangle{\nabla_N h,\nu})\psi^2 u \ge 0,\quad \forall \psi\in C^1(\Sigma).
	\end{multline}
	Taking $\psi=u^{-\frac12}$ and using \eqref{eq.lemm.spec2}, \eqref{eq.lemm.h}, we conclude that 
	\[\lambda |\Sigma|\le 2\int_ \Sigma K_\Sigma dA \le 8\pi\quad \Rightarrow\quad |\Sigma|\le \frac{8\pi}{\lambda}.\]
	Note that we have used Gauss--Bonnet, which also implies that $\Sigma$ is a $2$-sphere. The diameter upper bound follows from \cite[Lemma 16 and Lemma 18]{ChodoshLi2020bubble}.
\end{proof}

For the next lemma, recall that $r(x)=\dist_{\RR^{m}}(0,x)$.
\begin{lemm}\label{lemm.compare.distance} Below, $k\geq 2$ and $N^{k}$ is a compact connected manifold, possibly with boundary.  
\begin{enumerate}
\item Consider an immersion $N^{k}\to\RR^{m}\setminus\{0\}$. Consider $p,q\in N$ with $d_{\tilde g}(p,q)\le D$, where $\tilde g = r^{-2}g$ and $g$ is the induced metric on $N$. Then $r(p)\le e^Dr(q)$. 
\item Consider an immersion $\varphi : N^{k} \to \RR^{m}$ with $0 \in \varphi(N)$. Consider $p,q\in N\setminus\varphi^{-1}(0)$ with $d_{\tilde g}(p,q)\le D$. Write $g$ for the  for the induced metric on $N$ and let $\bar r(x) = d_{g}(\varphi^{-1}(0),x)$ denote the intrinsic distance on $N$. Then $\bar r(p) \leq e^{D} \bar r(q)$. 
\end{enumerate}
\end{lemm}

\begin{proof}
	We first establish (1). Choose a curve $\gamma: [0,L]\to N$, parametrized by $\tilde g$-unit speed, connecting $p$ and $q$, such that $L\le D+\eps$. Using $|\nabla r|_g \le 1$, we compute
	\begin{align*}
		\log r(q) - \log r(p) &= \int_0^L \frac{d}{dt} \log r(\gamma(t))dt\\
							  &= \int_0^L r(\gamma(t))^{-1} g(\nabla r, \gamma'(t)) dt \\
							  &\le \int_0^L r(\gamma(t))^{-1} |\nabla r|_g |\gamma'(t)|_g dt \\
							  &\le \int_0^L r(\gamma(t))^{-1} |\gamma'(t)|_g dt\\
							  &= \int_0^L |\gamma'(t)|_{\tilde g} dt =L\le D+\eps. 
	\end{align*}
	Thus $r(q)\le e^{D+\eps} r(p)$. The result follows by sending $\eps\to 0$.

For (2), we begin by noting that $|\nabla \bar r|_g=1$ and $r(x)\le \bar r(x)$ for any $x\in N$. Thus, arguing as above
	\[\log \bar r(q) - \log \bar r(p) \le \int_0^L \bar r(\gamma(t))^{-1} |\gamma'(t)|_g dt \le \int_0^L r(\gamma(t))^{-1} |\gamma'(t)|_g dt = L.\]
The proof is completed as above. 
\end{proof}

\begin{proof}[Proof of Theorem \ref{theo:main.global}]
	Let $r=\dist_{\RR^{4}}(\cdot, 0)$ and $\bar r= \dist_{M,g}(\cdot, 0)$, and consider $\tilde g = r^{-2}g$. Fix $\rho>0$, and consider the geodesic ball $B_{M}(0,e^{\frac{5\pi}{\sqrt{\lambda}}} \rho)$. By Proposition \ref{prop.one.ended}, $M\setminus B_{M}(0,e^{\frac{5\pi}{\sqrt \lambda}} \rho)$ has only one unbounded connected component $E$. Denote by $M' = M\setminus E$. %Then $\partial M'$ is connected. 
	We claim that $\partial M'=\partial E$ is connected. Indeed, since $M'$ and $E$ are both connected, if $\partial M'$ has more than one connected components, then one can find a loop in $M$ intersecting one component of $\partial M'$ exactly once, contradicting that $M$ is simply connected. Applying Lemma \ref{lemm:mu.bubble} to $(M'\setminus \{0\},\tilde g)$, we find a connected open set $\Omega$ in the $\frac{5\pi}{\sqrt \lambda}$ neighborhood of $\partial M'$, such that each connected component of $\partial \Omega\setminus \partial M'$ has area bounded by $\frac{8\pi}{\lambda}$ and diameter bounded by $\frac{2\pi}{\sqrt \lambda}$ (we emphasize here that the distance, area and diameter are with respect to $\tilde g$). Let $M_0$ be the connected component of $ M'\setminus \Omega$ that contains $0$.
	
	We make a few observations about $M_0$. First, we claim that $M\setminus M_0$ is connected. To see this, let $M_1$ be the union of connected components of $M'\setminus \Omega$ other than $M_0$. Then $M\setminus M_0 = M_1\cup \Omega\cup E$. Note that each connected component of $M_1$ share a common boundary with $\Omega$. Since $\Omega$ is connected, so is $M_1\cup \Omega$. Next, we claim that $M_0$ has only one boundary component: otherwise, since both $M_0$ and $M\setminus M_0$ are connected, as before we can find  a loop  in $M$ intersecting a connected component of $\partial M_0$ exactly once, contradicting that $M$ is simply connected.
	
	Denote by $\Sigma= \partial M_0$. By (2) in Lemma \ref{lemm.compare.distance}, $\min_{x\in \Sigma} \bar r(x)\ge \rho$. Since $B_M(0,\rho)$ is connected, this implies that $B_{M}(0,\rho) \subset M_0$. On the other hand, by comparing intrinsic to extrinsic distance, we see that $\max_{x \in \Sigma} r(x) \leq e^{\frac{5\pi}{\sqrt{\lambda}}} \rho$, so
	\[|\Sigma|_g =\int_\Sigma d\mu = \int_\Sigma r^2 d\tilde \mu \le e^{\frac{10\pi}{\sqrt{\lambda}}} \rho^2 |\Sigma|_{\tilde g}\le \frac{8\pi}{\lambda} e^{\frac{10\pi}{\sqrt{\lambda}}} \rho^2. \]
Thus, Corollary \ref{coro.isop.ineq} implies that
	\begin{equation*}
		|B_{M}(0,\rho)|_{g} \le |M_0|_{g} \le \frac{\|\Phi\|_{C^1}}{3\min_{\nu \in S^3} \Phi(\nu)} e^{\frac{5\pi}{\sqrt{\lambda}}} \rho |\partial M_0|_{g}
			 \le 
			 \frac{8\pi e^{\frac{15\pi}{\lambda}} \|\Phi\|_{C^1}}{3\lambda \min_{\nu \in S^3} \Phi(\nu)} \rho^3.
	\end{equation*}
	This proves the first part of the assertion. 
	
	Now consider a connected component $\Sigma_0$ of $\partial B_{M}(0,\rho)$, and let $E$ be the connected component of $M \setminus B_{M}(0,\rho)$ such that $\partial E$ contains $\Sigma_0$. Since $M$ is simply connected, we must have that $\partial E = \Sigma_0$. Apply Lemma \ref{lemm:mu.bubble} to $M\setminus E$, and obtain a  connected surface $\Sigma$ such that $\dist_{ \tilde g} (\Sigma_0,\Sigma)\le \frac{5\pi}{\sqrt \lambda}$, and $\diam_{\tilde g}(\Sigma)\le \frac{2\pi}{\sqrt \lambda}$. (The proof that $\Sigma$ is connected follows a similar argument as used above.) By the triangle inequality, we have that $\diam_{\tilde g} (\Sigma_0)\le \frac{7\pi}{\sqrt \lambda}$. Thus, Lemma \ref{lemm.compare.distance} implies that
	\[\max_{x\in \Sigma_0} r(x)\le e^{\frac{7\pi}{\sqrt \lambda}} \min_{x\in \Sigma_0}r(x).\]
This proves the assertion. 	
\end{proof}

\begin{proof}[Proof of Theorem \ref{theo:main.local}]
	The proof is very similar to that of Theorem \ref{theo:main.global}. We apply Lemma \ref{lemm:mu.bubble} to $(M\setminus \{0\}, \tilde g = r^{-2}g)$ and find a region $\Omega$ in the $\tfrac{5\pi}{\sqrt \lambda}$ neighborhood of $\partial M$, such that each connected component of $\Omega\setminus \partial M$ has area bounded by $\frac{8\pi}{\lambda}$ (again, the distance and area are with respect to $\tilde g$). Let $M'$ be the connected component of $M\setminus \Omega$ that contains $\{0\}$. Then $\partial M'$ is connected.
	
	Denote by $\Sigma = \partial M'$, and $\rho_0 = e^{-\frac{5\pi}{\sqrt \lambda}}$. By (1) in Lemma \ref{lemm.compare.distance}, $\min_{x\in \Omega} r(x)\ge \rho_0$. In particular, this implies that $M_{r_0}^*\subset M'$. We have
	\[|\Sigma|_{g} = \int_{\Sigma} d\mu = \int_\Sigma r^2 d\tilde \mu\le |\Sigma|_{\tilde g} \le \frac{8\pi}{\lambda}. \]
	Therefore, Corollary \ref{coro.isop.ineq} implies that
	\[|M_{\rho_0}^*|_{g} \le |M'| \le \frac{\|\Phi\|_{C^1}}{3\min_{\nu\in S^3} \Phi(\nu)} |\Sigma|_g \le \frac{8\pi  \|\Phi\|_{C^1}}{3\lambda\min_{\nu\in S^3}\Phi(\nu)}.\]
This completes the proof. 
\end{proof}

\begin{rema}
	In the more general case where we do not assume that $M$ is simply connected or has one end (or boundary component), similar proofs work out. The only modification here is that $\partial M_0$ in the proof of Theorem \ref{theo:main.global} (or $\partial M'$ in the proof of Theorem \ref{theo:main.local}) has connected components bounded by $b_1(M) + E$, where $E$ is the number of ends if $M$ is complete, and is the number of boundary components if $M\subset B_1(0)$. Thus, we have
	\[|B_{M,R}(0)|\le V_0 (b_1(M) + E),\]
	if $M$ is complete, and 
	\[|M_{\rho_0}^*| \le V_1 (b_1(M) + E),\]
	if $M\subset B_1(0)$.
\end{rema}

\appendix

\section{First and second variation}\label{app:var}
We derive first and second variations of $\bphi$ with emphasis on our geometric applications (see also \cite[Appendix A]{DePhilippisMaggi2017dimensional} and \cite[Section 2]{Winklmann2006note}). For $M^n\to \RR^{n+1}$ a two-sided immersion, set
\[
\bphi(M) : = \int_M \Phi(\nu) 
\]
for $\Phi : \RR^{n+1}\to (0,\infty)$ an elliptic integrand. 

\subsection{First variation}

Consider a $1$-parameter family of surfaces $M_t$ with normal speed at $t=0$ given by $u \nu$  (with $u \in C^{1}_{c}(M\setminus\partial M)$). Recall that $\dot \nu = - \nabla u$. We find
\begin{align*}
	\frac{d}{dt}\bigg\vert_{t=0}\bphi(M_t) & = \int_M \left( H u \Phi(\nu) - D_{\nabla u}\Phi(\nu) \right)\\
	& = \int_M \left( H  \Phi(\nu) + \Div_{M} ( D \Phi(\nu)^T) \right) u\\
	& = \int_M \left( H  \Phi(\nu) + \Div_M (D \Phi(\nu) - (D_{\nu} \Phi(\nu)) \nu)     \right) u\\
	& = \int_M \left( H  \Phi(\nu) + \Div_M (D \Phi(\nu)) -  (D_{\nu} \Phi(\nu))H   \right) u.
\end{align*}
Now, we note that we have that $D \Phi(\nu)\cdot \nu = \Phi(\nu)$ by Euler theorem for homogeneous functions. Thus, we find that 
\begin{equation}\label{eq:first-var}
\frac{d}{dt}\bigg\vert_{t=0}\bphi(M_t)  = \int_M   \Div_M( D \Phi(\nu)) u.
\end{equation}
Thus,
\begin{equation}\label{eq:HPhi}
H_\Phi = \Div_M( D \Phi(\nu)) . 
\end{equation}
vanishes if and only if $M$ is a critical point of $\bphi$. Let us rewrite this as follows (with $\{e_{i}\}_{i=1}^{n}$ a local orthonormal frame for $M$):
\begin{align*}
	\Div_M( D \Phi(\nu)) & = \sum_{i=1}^n (D_{e_i} D \Phi(\nu)) \cdot e_i\\
	& = \sum_{i=1}^n D^2\Phi(\nu)[D_{e_i}\nu,e_i]\\
	& = \sum_{i=1}^n D^2\Phi(\nu)[S_\Sigma(e_i),e_i],
\end{align*}
for $S_M$ the shape operator of $M$. Let us define $\Psi(\nu): T\RR^{n+1}\to T\RR^{n+1}$ by $\Psi(\nu): X \mapsto D^{2}\Phi(\nu)[X,\cdot]$. (This is just the $(1,1)$-tensor associated to $D^{2}\Phi(\nu)$ via the Euclidean metric.)

Then, we find
\begin{equation}\label{eq:Hphi-trace}
H_{\Phi} = \tr_{M}(\Psi(\nu) S_{M})
\end{equation}
Note that for $\Phi(\nu) = |\nu|$, we have
\[
D\Phi(\nu) = |\nu|^{-1}\nu, \Psi(\nu) = |\nu|^{-1} \Id - |\nu|^{-3}\nu\otimes \nu^{\flat}
\]
so in particular, when $|\nu|=1$, we find $\Psi(\nu)|_{T_{p}\Sigma} = \Id_{T_{p}\Sigma}$. Thus, this recovers the usual mean curvature.

\subsection{Second variation}\label{subsec:2nd-var}
Recall the tube formula:
\[
\dot S = - \nabla^{2}u - S^{2}u
\]
(where we are regarding $\nabla^{2}u$ as a $(1,1)$-tensor via $g_{M}$). Note also that the trace of a $(1,1)$-tensor is independent of the metric. Thus, we find
\[
\dot H_{\Phi} = \tr_{M} (-\Psi(\nu)\nabla^{2}u - \Psi(\nu) S_{M}^{2} u + \Psi(\nu)' S_{M})
\]
Note that
\[
\Psi(\nu)' = - (D_{\nabla u}\Psi)(\nu)
\]
Hence,
\[
\dot H_{\Phi} = \tr_{M} (-\Psi(\nu)\nabla^{2}u - \Psi(\nu) S_{M}^{2} u -  (D_{\nabla u}\Psi)(\nu) S_{M} ).
\]
Integration on $M$ gives
\begin{equation}
\frac{d^2}{dt^2}\bigg\vert_{t=0} \bphi(M_t) = \int_M \bangle{\nabla u, \Psi(\nu) \nabla u} - \tr_{M} \left(\Psi(\nu) S_M^2\right) u^2.
\end{equation}
Thus, stability implies that
\begin{equation}
\int_M |\nabla u|^2 - \Lambda |A|^2 u^2 \ge 0,\quad \forall u\in C_c^1(M\setminus\partial M).
\end{equation}
Here $\Lambda$ depends on ellipticity of $\Phi$. In particular, if $\Phi$ satisfies \eqref{assumption.phi}, then \eqref{eq.second.variation} implies that for $\Phi$-stable surfaces $M$, we have
\begin{equation}
\int_M |\nabla u|^2 - \frac{1}{\sqrt 2} |A_M|^2 u^2\ge 0, \quad \forall u\in C_c^1(M\setminus\partial M).
\end{equation}

Note that when $\Phi(X) = |X|$, we have seen that $\Psi(Y) = |Y|^{-1}\Id - |Y|^{-3}Y \otimes Y^{\flat}$. Hence, 
\[
D_{X}\Psi(\nu) = \frac{d}{dt}\Big|_{t=0} \Psi(\nu + X) = 2 (X\cdot \nu)\Id - X\otimes \nu^{\flat}  - \nu \otimes X^{\flat}
\]
In particular, $D_{\nabla u} \Psi(\nu)|_{T_{p}M} = 0$. Thus, we recover the standard second variation formula in this case.

\subsection{First variation through vector fields} \label{subsec:first-var-vf}

We also deduce first variation formula of $\bphi$ through variations that are not necessarily normal to $M$. We compute as follows:
\begin{align*}
	& \int_M \Phi(\nu) \Div_M X \\
	& = \int_M \Phi(\nu) \Div_M X^T + \Phi(\nu) (X\cdot \nu) H\\
	& = \int_M  \Div_M(\Phi(\nu) X^T) - \nabla(\Phi(\nu))\cdot X^T + \Phi(\nu) (X\cdot \nu) H\\
	& = \int_M  \Div_M(\Phi(\nu) X^T) - D_{D\Phi(\nu)^T}\nu \cdot X^T + \Phi(\nu) (X\cdot \nu) H\\
	& = \int_M  \Div_M(\Phi(\nu) X^T) - D_{D\Phi(\nu)^T}X^T \cdot \nu + \Phi(\nu) (X\cdot \nu) H\\
	& = \int_M  \Div_M(\Phi(\nu) X^T) + D_{D\Phi(\nu)^T}(X\cdot \nu)  - D_{D\Phi(\nu)^T}X \cdot \nu + \Phi(\nu) (X\cdot \nu) H\\
	& = \int_M  \Div_M(\Phi(\nu) X^T) + \Div((X\cdot\nu) D\Phi(\nu)^T) - (X\cdot \nu) \Div_\Sigma D\Phi(\nu)^T\\
	&\qquad  - D_{D\Phi(\nu)^T}X \cdot \nu + \Phi(\nu) (X\cdot \nu) H\\
	& = \int_M   - (X\cdot \nu) \Div_M  D\Phi(\nu)^T  - D_{D\Phi(\nu)^T}X \cdot \nu + \Phi(\nu) (X\cdot \nu) H  \\
	&\qquad + \int_{\partial M} \Phi(\nu) X\cdot \eta + (X\cdot\nu) D\Phi(\nu)\cdot \eta\\
	& = \int_M   - (X\cdot \nu) \Div_M  D\Phi(\nu)    - D_{D\Phi(\nu)^T}X \cdot \nu  + \int_{\partial M} \Phi(\nu) X\cdot \eta + (X\cdot\nu) D\Phi(\nu)\cdot \eta.
\end{align*}
Thus, we find that if $H_\Phi = 0$, then 
\begin{equation}
\int_M \Phi(\nu)\Div_M X +   D_{D\Phi(\nu)^T}X \cdot \nu  = \int_{\partial M} \Phi(\nu) X\cdot \eta + (X\cdot\nu) D\Phi(\nu)\cdot \eta.
\end{equation}

\section{Some computations for quadratic forms}\label{appendix.quadratic.forms}

In this section we explicitly compute the constant $c_0$ appeared in Lemma \ref{lemm:compare.A.and.R}. The approach is elementary.

\begin{lemm}\label{lemm:quadratic.forms}
	Let $a_1\le a_2\le a_3$ be positive constants such that  $\frac{a_3}{a_1}\le \sqrt 2$.
	Consider quadratic forms
	\[Q_1(k_1,k_2) = \frac{a_1^2 + a_3^2}{a_3^2} k_1^2 + \frac{2a_1a_2}{a_3^2} k_1k_2 + \frac{a_2^2+a_3^2}{a_3^2} k_2^2,\]
	\[Q_2(k_1,k_2) = \frac{2a_1}{a_3}k_1^2 + \frac{2(a_1+a_2-a_3)}{a_3} k_1k_2 + \frac{2a_2}{a_3}k_2^2.\]
	Then we have $Q_1 \le  c_0 Q_2$, where 
	\[c_0 = \frac{1}{\sqrt 2 - \frac12 } \approx 1.09.\]
	
\end{lemm}

\begin{proof}
	Write $\alpha = \frac{a_1}{a_3}$, $\beta = \frac{a_2}{a_3}$, with $2^{-\frac 12}\le \alpha \le \beta\le 1$. Then
	\begin{multline}
		Q_1(k_1,k_2) = (1+\alpha^2 )k_1^2 + 2\alpha \beta k_1k_2 + (1+\beta^2) k_2^2 \\
		= (1+\alpha^2) \left(k_1 + \frac{\alpha\beta}{1+\alpha^2} k_2\right)^2 + \frac{1+\alpha^2+\beta^2}{1+\alpha^2} k_2^2.
	\end{multline}
	Under the substitution $x= k_1 + \frac{\alpha\beta}{1+\alpha^2} k_2$, $y= k_2$, we have $k_1 + k_2 + (-\alpha k_1 -\beta k_2)=(1-\alpha) x + \frac{1- \beta -\alpha\beta +\alpha^2}{1+\alpha^2} y$.
	Thus, by Cauchy-Schwartz,
	\begin{align*}
		(Q_1-Q_2)(k_1,k_2) & = (k_1+k_2-\alpha k_1 -\beta k_2)^2 \\
										& = \left((1-\alpha) x+ \frac{1 -\beta -\alpha \beta +\alpha^2}{1+\alpha^2}y\right)^2\\
										& \le c_1\left((1+\alpha^2)x^2 +\frac{1+\alpha^2 +\beta^2}{1+\alpha^2}y^2\right) = c_1 Q_1(k_1,k_2),
	\end{align*}
	where $c_1= \frac{(1-\alpha)^2}{1+\alpha^2} + \left(\frac{1-\beta -\alpha\beta +\alpha^2}{1+\alpha^2}\right)^2 \cdot\frac{1+\alpha^2}{1+\alpha^2+\beta^2}$. This gives $Q_1 \le \frac{1}{1-c_1}Q_2$. Using $2^{-\frac 12}\le \alpha\le\beta \le 1$, we have:
	\[c_1 \le \frac{(1-2^{-\frac 12})^2}{1+ \frac 12} + \left(\frac{1-2^{-\frac 12}}{1+\frac 12}\right)^2 \cdot \frac{1+ \frac 12}{2} = \frac32 - \sqrt 2.\]
	The result follows.
\end{proof}

\bibliographystyle{plain}
\bibliography{bib}

\end{document}